\documentclass{amsart}
\title{Elementary formulas for integer partitions}
\usepackage{amssymb,amsmath,amsthm,epsfig,graphics,latexsym,enumerate}

 \theoremstyle{plain}
 \newtheorem{theorem}{Theorem}

 \theoremstyle{definition}
 \newtheorem{definition}{Definition}

 \theoremstyle{remark}

 \numberwithin{equation}{section}

\begin{document}

 \author{Mohamed El Bachraoui}
 \address{Mohamed El Bachraoui, Dept. Math. Sci,
 United Arab Emirates University, PO Box 17551, Al-Ain, UAE}
 \email{melbachraoui@uaeu.ac.ae}

  \begin{abstract}
  In this note we will give various exact formulas for functions on integer partitions including the functions
  $p(n)$ and $p(n,k)$ of the number of partitions of $n$ and the number of such partitions into exactly $k$ parts respectively.
   For instance, we shall prove that
   \[
  \begin{split}
 p(n) &=
 \sum_{d|n}
 \sum_{k=1}^{d}
 \sum_{i_0 =1}^{\lfloor d/k \rfloor}
 \sum_{i_1 =i_0}^{\lfloor\frac{d- i_0}{k-1} \rfloor}
 \sum_{i_2 =i_1}^{\lfloor\frac{d- i_0 - i_1}{k-2} \rfloor} \ldots
 \sum_{i_{k-3}=i_{k-4}}^{\lfloor\frac{n- i_0 - i_1-i_2- \ldots-i_{k-4}}{3} \rfloor} \\
 & \sum_{c|(d,i_0,i_1,i_2,\ldots,i_{k-3})}
 \mu(c) \left(\left\lfloor \frac{d-i_0-i_1-i_2- \ldots i_{k-3}}{2c} \right\rfloor -
 \left\lfloor\frac{i_{k-3}-1}{c} \right\rfloor \right).
 \end{split}
 \]
  Our proofs are elementary.
 \end{abstract}
 \date{\textit{\today}}
  \maketitle
  \section{Introduction}
  Among challenges that faced mathematicians who interested in integer partitions
  was the problem to find a formula
   to compute the number of partitions of any positive integer.
   Hardy and Ramanujan in \cite{Hardy-Ramanujan} gave the following asymptotic formula for $p(n)$,
   \[
   p(n) \sim \frac{e^{\pi\sqrt{2n/3}}}{4n \sqrt{3}},
   \]
   and Rademacher in \cite{Rademacher}
   gave the following exact formula for $p(n)$,
  \begin{equation} \label{exact1}
  p(n)= \frac{1}{\pi \sqrt{2}}\sum_{k\geq 1}A_k(n) \sqrt{k}\left[ \frac{d}{dx}
    \frac{\sinh \left( \frac{\pi}{k}\sqrt{\frac{2}{3}(x- \frac{1}{24})} \right)}{\sqrt{x-\frac{1}{24}}}
    \right]_{x=n},
  \end{equation}
  where
  \[
   A_k(n)=\sum_{\substack{h \bmod k \\ (h,k)=1}}\omega_{h,k} e^{-2\pi i n h/k}
  \]
  and $\omega_{h,k}$ is a certain $24$th root of unity. To find such deep formulas the authors used tools from complex analysis.
  A standard reference for more details about integer partitions is \cite{Andrews1}.
  Our purpose in this work is
  to give exacts formulas involving only finite sums for functions on integer partitions.

  A nonempty finite set $A$ of positive integers is \emph{relatively prime} if $\gcd(A)=1$ and it is \emph{relatively prime to $m$}
    if $\gcd(A,m)=1$. Accordingly, a partition of $n$ is called relatively prime if its parts form a relatively prime set and it is called relatively prime to $m$ if its parts form a set which is relatively prime to $m$. Throughout let $k,l,m,n$ be positive integers, let
   $\lfloor x \rfloor$ be the floor of $x$, and let $\mu(n)$ be the M\"obius mu function.
  \begin{definition}
  Let
  $p(n)$ be  the number of (unrestricted) partitions of $n$, let
  $p_{\Psi(m)}(n)$ be the number of partitions of $n$ which are relatively prime to $m$, and
  let $p_{\Psi}(n)$ be the number of relatively prime partitions of $n$. Let
  $p(n,k)$ be the number of partitions of $n$ into exactly $k$ parts, let
  $p_{\Psi(m)}(n,k)$ be the number of partitions of $n$ into exactly $k$ parts which are relatively prime to $m$, and
  let $p_{\Psi}(n,k)$ be the number of relatively prime partitions of $n$ into exactly $k$ parts.
  Let $p(n,k,l)$ be the number of partitions of $n$ into exactly $k$ parts the smallest of which is $l$, let
  $p_{\Psi(m)}(n,k,l)$ be the number of partitions of $n$ into exactly $k$ parts which are relatively prime to $m$ with smallest
  part $l$, and let
  $p_{\Psi}(n,k,l)$ be the number of relatively prime partitions of $n$ into exactly $k$ parts the smallest
  of which is $l$.
  Let $p(n,k,\geq l)$ be the number of partitions of $n$ into exactly $k$ parts none of which is smaller than $l$, let
  $p_{\Psi(m)}(n,k, \geq l)$ be the number of partitions of $n$ into exactly $k$ parts which are relatively prime to $m$ and with
  no part smaller than $l$, and let
  $p_{\Psi}(n,k, \geq l)$ be the number of relatively prime partitions of $n$ into exactly $k$ parts none of which is smaller than $l$.
  \end{definition}
  \noindent
  We now list few identities relating some of these functions. Easy proofs are omitted.
  \begin{theorem} \label{identities-list}
  We have
  \begin{enumerate}
  \item $p_{\Psi(n)}(n)= p_{\Psi}(n)$, $p_{\Psi(n)}(n,k)= p_{\Psi}(n,k)$, and $p_{\Psi(n)}(n,k,l) = p_{\Psi}(n,k,l)$.
  \item $p(n) = \sum_{d|n} p_{\Psi} (d)$ or equivalently $p_{\Psi}(n)=\sum_{d|n}\mu(d) p(n/d)$.
  \item $p(n,k) = \sum_{d|n} p_{\Psi} (d,k)$ or equivalently $p_{\Psi}(n,k)=\sum_{d|n}\mu(d) p(n/d,k)$.
  \item $p(n,k,l) = \sum_{d|n} p_{\Psi}(n/d,k,l/d)$ or equivalently $p_{\Psi}(n,k,l)=\sum_{d|n}\mu(d) p(n/d,k,l/d)$.
  \item $p(n) = \sum_{k=1}^n p(n,k)$ and $p_{\Psi(m)}(n) = \sum_{k=1}^n p_{\Psi(m)}(n,k)$.
  \item $p(n,k) = \sum_{l=1}^{\lfloor n/k \rfloor} p(n,k,l)$ and
        $p_{\Psi(m)}(n,k) = \sum_{l=1}^{\lfloor n/k \rfloor} p_{\Psi(m)}(n,k,l)$.
  \item If $k> 1$, then $p_{\Psi(m)}(n,k,l) = p_{\Psi(m,l)}(n-l,k-1, \geq l)$.
  \item If $l\leq \lfloor n/k \rfloor$, then $p(n,k,\geq l) = \sum_{j=l}^{\lfloor n/k \rfloor} p(n,k,j)$
        and $p_{\Psi(m)}(n,k,\geq l) = \sum_{j=l}^{\lfloor n/k \rfloor} p_{\Psi(m)}(n,k,j)$.
  \end{enumerate}
  \end{theorem}
 Note that the equivalence of the two identities in Theorem \ref{identities-list} (4) follows by the M\"obius inversion
 formula for
 arithmetical functions of several variables, see \cite[Theorem 2]{ElBachraoui1}.
 Further it is understood that
  \[
  p(n,k)=p_{\Psi(m)}(n,k) = 0,\ \text{if\ } k > n
  \]
  and
  \[
  p(n,k,l)=p_{\Psi(m)}(n,k,l) = p(n,k,\geq l)= p_{\Psi(m)}(n,k,\geq l) =0,\ \text{if $k > n$ or $l > \lfloor n/k \rfloor$.}
  \]
\noindent
  The following result is crucial to our formulas.
  \begin{theorem}[\cite{Ayad-Kihel1, ElBachraoui-Salim}] \label{phi-interval}
  Let $a$ and $b$ be positive integers such that $a \leq b$ and let
  \[
  \Phi([a,b],n) = \# \{c \in \{a,a+1,\ldots,b\}:\ \gcd(c,n)=1\}.
  \]
  Then
  \[
  \Phi([a,b],n) = \sum_{d|n}\mu(d)(\lfloor b/d\rfloor - \lfloor (a-1)/d\rfloor ).
  \]
  \end{theorem}
  Note that this result generalizes the Euler phi function since 
  \[\Phi(n)=\#\{c\in [1,n]:\ \gcd(c,n)=1\} =\Phi([1,n],n). \]
 \section{Formulas for $p_{\Psi(m)}(n,k,l)$ and $p_{\Psi}(n,k,l)$}
 \begin{theorem} \label{base-case}
 If $n\geq 2$ and $l\leq \lfloor n/2 \rfloor$, then
 \[
 p_{\Psi(m)}(n,2,\geq l) = \sum_{d|(n,m)}\mu(d)\left(\left\lfloor \frac{n}{2d}\right\rfloor - \left\lfloor\frac{l-1}{d}\right\rfloor\right).
 \]
 \end{theorem}
 \begin{proof}
 We have
 \[
 \begin{split}
 p_{\Psi(m)}(n,2,\geq l)
 &=
 \# \{ a\in [l, \lfloor n/2 \rfloor ]:\ \gcd(a,n-a,m)=1 \} \\
 &=
 \# \{ a\in [l, \lfloor n/2 \rfloor ]:\ \gcd(a,(n,m))=1 \} \\
 &=
 \Phi([l, \lfloor n/2 \rfloor ], \gcd(n,m) ) \\
 &=
 \sum_{d|(n,m)} \mu(d)\left( \left\lfloor \frac{n}{2d} \right\rfloor - \left\lfloor \frac{l-1}{d} \right\rfloor \right),
 \end{split}
 \]
 where the last identity follows by Theorem \ref{phi-interval}.
 \end{proof}
 %
 %
 %
 \begin{theorem} \label{main1}
 We have
 \[
 \begin{split}
 (a)\quad p_{\Psi(m)}(n,k,i_0) &=
 \sum_{i_1 =i_0}^{\lfloor\frac{n- i_0}{k-1} \rfloor}
 \sum_{i_2 =i_1}^{\lfloor\frac{n- i_0 - i_1}{k-2} \rfloor} \ldots
 \sum_{i_{k-3}=i_{k-4}}^{\lfloor\frac{n- i_0 - i_1-i_2- \ldots-i_{k-4}}{3} \rfloor}
  \sum_{d|(n,m,i_0,i_1,i_2,\ldots,i_{k-3})} \\
 & \mu(d) \left(\left\lfloor \frac{n-i_0-i_1-i_2- \ldots i_{k-3}}{2d} \right\rfloor -
 \left\lfloor\frac{i_{k-3}-1}{d} \right\rfloor \right). \\
 (b)\quad p_{\Psi}(n,k,i_0) &=
 \sum_{i_1 =i_0}^{\lfloor\frac{n- i_0}{k-1} \rfloor}
 \sum_{i_2 =i_1}^{\lfloor\frac{n- i_0 - i_1}{k-2} \rfloor} \ldots
 \sum_{i_{k-3}=i_{k-4}}^{\lfloor\frac{n- i_0 - i_1-i_2- \ldots-i_{k-4}}{3} \rfloor}
  \sum_{d|(n,i_0,i_1,i_2,\ldots,i_{k-3})} \\
 & \mu(d) \left(\left\lfloor \frac{n-i_0-i_1-i_2- \ldots i_{k-3}}{2d} \right\rfloor -
 \left\lfloor\frac{i_{k-3}-1}{d} \right\rfloor \right).
 \end{split}
 \]
 \end{theorem}
 \begin{proof}
 (a) Repeatedly application of Theorem \ref{identities-list} (7, 8) yields
 \[
 \begin{split}
 p_{\Psi(m)} (n,k,i_0) &= p_{\Psi(m,i_0)}(n- i_0,k-1,\geq i_0) \\
 &= \sum_{i_1 =i_0}^{\lfloor \frac{n-i_0}{k-1}\rfloor} p_{\Psi(m,i_0)}(n-i_0,k-1,i_1) \\
 &= \sum_{i_1 =i_0}^{\lfloor \frac{n-i_0}{k-1}\rfloor}
 \sum_{i_2=i_1}^{\lfloor\frac{n-i_0-i_1}{k-2}\rfloor} p_{\Psi(m,i_0,i_1)}(n-i_0-i_1,k-2, i_2) \\
 &= \sum_{i_1 =i_0}^{\lfloor \frac{n-i_0}{k-1}\rfloor}
  \sum_{i_2=i_1}^{\lfloor\frac{n-i_0-i_1}{k-2}\rfloor} \ldots
  \sum_{i_{k-3}=i_{k-4}}^{\lfloor\frac{n-i_0-i_1-\ldots-i_{k-4}}{3}\rfloor} \\
 & p_{\Psi(m,i_0,i_1,\ldots,i_{k-4})}(n-i_0-i_1-\ldots-i_{k-4},3,i_{k-3}) \\
 &= \sum_{i_1 =i_0}^{\lfloor \frac{n-i_0}{k-1}\rfloor}
  \sum_{i_2=i_1}^{\lfloor\frac{n-i_0-i_1}{k-2}\rfloor} \ldots
  \sum_{i_{k-3}=i_{k-4}}^{\lfloor\frac{n-i_0-i_1-\ldots-i_{k-4}}{3}\rfloor} \\
 & p_{\Psi(m,i_0,i_1,\ldots,i_{k-4},i_{k-3})}(n-i_0-i_1-\ldots-i_{k-4}-i_{k-3},2,\geq i_{k-3}) \\
 &= \sum_{i_1 =i_0}^{\lfloor \frac{n-i_0}{k-1}\rfloor}
  \sum_{i_2=i_1}^{\lfloor\frac{n-i_0-i_1}{k-2}\rfloor} \ldots
  \sum_{i_{k-3}=i_{k-4}}^{\lfloor\frac{n-i_0-i_1-\ldots-i_{k-4}}{3}\rfloor}
  \sum_{d|(n,m,i_0,i_1,i_2,\ldots,i_{k-3})} \\
 & \mu(d) \left(\left\lfloor \frac{n-i_0-i_1-i_2- \ldots i_{k-3}}{2d} \right\rfloor -
 \left\lfloor\frac{i_{k-3}-1}{d} \right\rfloor \right),
 \end{split}
 \]
 where the last identity follows by Theorem \ref{base-case}.

 \noindent
 (b) This part follows directly from part (b) since $p_{\Psi} (n,k,i_0) = p_{\Psi(n)} (n,k,i_0)$ by
  Theorem \ref{identities-list} (1).
 \end{proof}
 \section{Formulas for $p_{\Psi(m)}(n,k)$, $p_{\Psi}(n,k)$, and $p(n,k)$}
 \begin{theorem} \label{main2}
 We have
 \[
 \begin{split}
 (a)\quad p_{\Psi(m)}(n,k)&=
 \sum_{i_0 =1}^{\lfloor n/k \rfloor}
 \sum_{i_1 =i_0}^{\lfloor\frac{n- i_0}{k-1} \rfloor}
 \sum_{i_2 =i_1}^{\lfloor\frac{n- i_0 - i_1}{k-2} \rfloor} \ldots
 \sum_{i_{k-3}=i_{k-4}}^{\lfloor\frac{n- i_0 - i_1-i_2- \ldots-i_{k-4}}{3} \rfloor}
  \sum_{d|(n,m,i_0,i_1,i_2,\ldots,i_{k-3})} \\
 & \mu(d) \left(\left\lfloor \frac{n-i_0-i_1-i_2- \ldots i_{k-3}}{2d} \right\rfloor -
 \left\lfloor\frac{i_{k-3}-1}{d} \right\rfloor \right).
 \end{split}
 \]
 \[
 \begin{split}
 (b)\quad  p_{\Psi}(n,k)&=
 \sum_{i_0 =1}^{\lfloor n/k \rfloor}
 \sum_{i_1 =i_0}^{\lfloor\frac{n- i_0}{k-1} \rfloor}
 \sum_{i_2 =i_1}^{\lfloor\frac{n- i_0 - i_1}{k-2} \rfloor} \ldots
 \sum_{i_{k-3}=i_{k-4}}^{\lfloor\frac{n- i_0 - i_1-i_2- \ldots-i_{k-4}}{3} \rfloor}
  \sum_{d|(n,i_0,i_1,i_2,\ldots,i_{k-3})} \\
 & \mu(d) \left(\left\lfloor \frac{n-i_0-i_1-i_2- \ldots i_{k-3}}{2d} \right\rfloor -
 \left\lfloor\frac{i_{k-3}-1}{d} \right\rfloor \right).
 \end{split}
 \]
 \end{theorem}
 \begin{proof}
 Part (a) follows by Theorem \ref{identities-list} (6) and Theorem \ref{main1}.
 Part (b) follows by part (a) and Theorem \ref{identities-list} (1).
 \end{proof}
 \begin{theorem} \label{main3}
 We have
 \[
  \begin{split}
 p(n,k) &=
 \sum_{d|n}
 \sum_{i_0 =1}^{\lfloor d/k \rfloor}
 \sum_{i_1 =i_0}^{\lfloor\frac{d- i_0}{k-1} \rfloor}
 \sum_{i_2 =i_1}^{\lfloor\frac{d- i_0 - i_1}{k-2} \rfloor} \ldots
 \sum_{i_{k-3}=i_{k-4}}^{\lfloor\frac{d- i_0 - i_1-i_2- \ldots-i_{k-4}}{3} \rfloor}
  \sum_{c|(d,i_0,i_1,i_2,\ldots,i_{k-3})} \\
 & \mu(c) \left(\left\lfloor \frac{d-i_0-i_1-i_2- \ldots i_{k-3}}{2c} \right\rfloor -
 \left\lfloor\frac{i_{k-3}-1}{c} \right\rfloor \right).
 \end{split}
 \]
 \end{theorem}
 \begin{proof}
 The is a consequence of Theorem \ref{identities-list} (3) and Theorem \ref{main2}.
 \end{proof}
 \section{Formulas for $p_{\Psi(m)}(n)$, $p_{\Psi}(n)$, and $p(n)$}
 \begin{theorem} \label{main4}
 We have
 \[
 \begin{split}
 (a)\quad p_{\Psi(m)}(n)&=
 \sum_{k=1}^{n}
 \sum_{i_0 =1}^{\lfloor n/k \rfloor}
 \sum_{i_1 =i_0}^{\lfloor\frac{n- i_0}{k-1} \rfloor}
 \sum_{i_2 =i_1}^{\lfloor\frac{n- i_0 - i_1}{k-2} \rfloor} \ldots
 \sum_{i_{k-3}=i_{k-4}}^{\lfloor\frac{n- i_0 - i_1-i_2- \ldots-i_{k-4}}{3} \rfloor}
  \sum_{d|(n,m,i_0,i_1,i_2,\ldots,i_{k-3})} \\
 & \mu(d) \left(\left\lfloor \frac{n-i_0-i_1-i_2- \ldots i_{k-3}}{2d} \right\rfloor -
 \left\lfloor\frac{i_{k-3}-1}{d} \right\rfloor \right).
 \end{split}
 \]
 \[
 \begin{split}
 (b)\quad p_{\Psi}(n)&=
 \sum_{k=1}^{n}
 \sum_{i_0 =1}^{\lfloor n/k \rfloor}
 \sum_{i_1 =i_0}^{\lfloor\frac{n- i_0}{k-1} \rfloor}
 \sum_{i_2 =i_1}^{\lfloor\frac{n- i_0 - i_1}{k-2} \rfloor} \ldots
 \sum_{i_{k-3}=i_{k-4}}^{\lfloor\frac{n- i_0 - i_1-i_2- \ldots-i_{k-4}}{3} \rfloor}
  \sum_{d|(n,i_0,i_1,i_2,\ldots,i_{k-3})} \\
 & \mu(d) \left(\left\lfloor \frac{n-i_0-i_1-i_2- \ldots i_{k-3}}{2d} \right\rfloor -
 \left\lfloor\frac{i_{k-3}-1}{d} \right\rfloor \right).
 \end{split}
 \]
 \end{theorem}
 \begin{proof}
 Combine Theorem \ref{identities-list} (5) with Theorem \ref{main2} to obtain part (a). As to part (b)
 combine part (a) with Theorem \ref{identities-list} (1).
 \end{proof}
 \begin{theorem} \label{main5}
 We have
 \[
  \begin{split}
 p(n) &=
 \sum_{d|n}
 \sum_{k=1}^{d}
 \sum_{i_0 =1}^{\lfloor d/k \rfloor}
 \sum_{i_1 =i_0}^{\lfloor\frac{d- i_0}{k-1} \rfloor}
 \sum_{i_2 =i_1}^{\lfloor\frac{d- i_0 - i_1}{k-2} \rfloor} \ldots
 \sum_{i_{k-3}=i_{k-4}}^{\lfloor\frac{n- i_0 - i_1-i_2- \ldots-i_{k-4}}{3} \rfloor}
  \sum_{c|(d,i_0,i_1,i_2,\ldots,i_{k-3})} \\
 & \mu(c) \left(\left\lfloor \frac{d-i_0-i_1-i_2- \ldots i_{k-3}}{2c} \right\rfloor -
 \left\lfloor\frac{i_{k-3}-1}{c} \right\rfloor \right).
 \end{split}
 \]
 \end{theorem}
 \begin{proof}
 Use Theorem \ref{identities-list} (2) and Theorem \ref{main4}.
 \end{proof}


\begin{thebibliography}{9}

 \bibitem{Ayad-Kihel1}
 Mohamed Ayad and Omar Kihel,
 \emph{On the Number of Subsets Relatively Prime to an Integer},
 Journal of Integer Sequences, {\bf Vol. 11}, (2008), Article 08.5.5.

 \bibitem{Andrews1}
  George E. Andrews,
  \emph{The theory of partitions}, Cambridge University Press, 1998.

 \bibitem{ElBachraoui1}
 Mohamed El Bachraoui,
 \emph{The number of relatively prime subsets and phi functions for sets
 $\{m,m+1,\ldots,n\}$},
 Integers {\bf 7}
 (2007), A43, 8pp.



  \bibitem{ElBachraoui-Salim}
 Mohamed El Bachraoui and Mohamed Salim,
 \emph{Combinatorial Identities Involving the M\"obius Function},
 Submitted. (\emph{Available on}: arXiv:0909.2983.)


 \bibitem{Hardy-Ramanujan}
 Hardy, G. H. and Ramanujan, S.
 \emph{Asymptotic Formulae in Combinatory Analysis}, Proc. London Math. Soc. (2) 17, 75-115, 1918.


 \bibitem{Rademacher}
 Rademacher, H.
 \emph{On the Partition Function $p(n)$},
 Proc. London Math. Soc. (2) 43, 241-254, 1937.




 %
 %
\end{thebibliography}
\end{document}